\documentclass[draft,reqno]{amsart}

\usepackage{amsmath}
\usepackage{amssymb}
\usepackage[all]{xy}
\usepackage{picinpar}
\usepackage{palatino}
\usepackage[usenames,dvipsnames,svgnames,table]{xcolor}
\usepackage{mathtools}

\def\eu{\mathfrak}
\def\ma{\mathbb}
\def\mc{\mathcal}

\def\p{{\mc P}_{\infty}}
\def\f{{\ma F}_q^{\ast}}
\def\F{{\ma F}_q}
\def\pK{\eu p}
\def\pL{\eu P}
\def\fin{\hfill\qed\bigskip}
\def\*#1{#1^*}
\def\lra{\longrightarrow}

\def\Ku#1#2{k\big(\sqrt[#1]{#2}\big)}
\def\produ{\big\{(\alpha_1, \ldots, \alpha_t)\in \prod_{i=1}^t
K_{{\mathfrak p}_i}\mid \prod_{i=1}^t \N_{{K_{{\mathfrak p}_i}
/{\mathbb R}}}(\alpha_i)\in {{\mathbb R}^*}^2\big\}}
\def\producto#1{\big\{(\alpha_1, \ldots, \alpha_t)\in \prod_{i=1}^t
K_{{\mathfrak p}_i}\mid \prod_{i=1}^t \N_{{K_{{\mathfrak p}_i}
/k_{\infty}}}(\alpha_i)\in {k_{\infty}^{*#1}\big\}}}
\def\A{\Delta\times\prod_{\pK\nmid\infty}U_{\pK}}
\def\B{\prod_{\pK|\infty} K_{\pK}^*\times\prod_{\pK\nmid\infty}U_{\pK}}

\newcommand{\Gal}{\operatorname{Gal}}

\newcommand{\N}{\operatorname{N}}
\newcommand{\rest}{\operatorname{rest}}

\newcounter{bean}

\def\las{\begin{list}
	{{\rm {(\arabic{bean})}}}{\usecounter{bean}
\setlength{\labelwidth}{0.8in}
\setlength{\labelsep}{0.3cm}
\setlength{\leftmargin}{1cm}}}

\numberwithin{equation}{section}
\newtheorem{theorem}{Theorem}[section]
\newtheorem{proposition}[theorem]{Proposition}
\newtheorem{lemma}[theorem]{Lemma}
\newtheorem{examples}[theorem]{Examples}
\newtheorem{remark}[theorem]{Remark}
\newtheorem{definition}[theorem]{Definition}
\newtheorem{corollary}[theorem]{Corollary}

\title[A note on extended genus fields of Kummer extensions]
{A note on extended genus fields of Kummer extensions of
rational function fields}
 
\author[M. Rzedowski]{Martha Rzedowski--Calder\'on}
\address{Departamento de Control Autom\'atico\\
Centro de Investigaci\'on y de Estudios Avanzados del I.P.N.}
\email{mrzedowski@ctrl.cinvestav.mx}

\author[G. Villa]{Gabriel Villa--Salvador}
\address{Departamento de Control Autom\'atico\\
Centro de Investigaci\'on y de Estudios Avanzados del I.P.N.}
\email{gvillasalvador@gmail.com, gvilla@ctrl.cinvestav.mx}



\date{March 1st., 2024}

\begin{document}

\begin{abstract}

We consider the generalization of the extended genus
field of a prime degree 
cyclic Kummer extension of a rational function field
obtained by R. Clement in 1992 to general Kummer extensions.
We observe that the same approach of Clement works in general.

\end{abstract}

\maketitle

\section{Introduction}\label{S1}

The first one to consider genus fields, or more precisely
extended genus fields,
of function fields was R. Clement
in \cite{Cle92}. She considered a cyclic Kummer extension $K$
of $k:={\mathbb F}_q(T)$, a global rational function field of prime
degree $l$, where $l|q-1$. She used the ideas of H. Hasse developed
in \cite{Has51}, where he considered a quadratic extension $K$ of
the field of rational numbers ${\mathbb Q}$.

In the case of number fields, both, the genus and the extended
genus fields of a number field $K$ are canonically defined as the
maximal extension of $K$ of the form $KF$ contained in the
Hilbert class and in the extended Hilbert class field, respectively,
where $F$ is an abelian extension of ${\mathbb Q}$. The Hilbert
class field (resp. extended Hilbert class field) of $K$ is defined
as the maximal unramified (resp. maximal 
unramified at the finite primes) abelian extension of $K$. 

In the case of global function fields
the situation is quite different. There are several reasonable possible 
definitions of Hilbert class field of a function field $K$. The main
reason is that, in the number field case, the Hilbert class
field is a finite extension of $K$. In the case of a function field $K$,
the extensions of constants of $K$ are unramified so that the
maximal unramified abelian of $K$ is of infinite degree.

For a quadratic number field $K$ of ${\mathbb Q}$, the
extended Hilbert class field $K_{H^+}$ of $K$ is the abelian extension such
that the fully decomposed primes of $K$ in $K_{H^+}$ are precisely
the nonzero principal prime ideals of $K$ generated by a totally
positive element. In the case of a quadratic field $K$ of ${\mathbb
Q}$, an element is totally positive if and only if its norm in ${\mathbb
Q}$ is the square of a real number. In other words, the norm group
in the id\`ele group $J_K$ of $K$ is $K^*(\Delta\times \prod_p U_p)$,
where $p$ runs through the prime numbers, $U_p$ denotes the
group of units of ${\mathbb Q}_p$ and $\Delta=\produ$
where $\{\pK_i\}_{i=1}^t$ is the set of infinite primes in $K$ ($t=1$
or $2$ and $K_{\pK_i}\in\{{\mathbb R},{\mathbb C}\}$).

For a cyclic Kummer extension $K/k$ of prime degree
$l$, where $k={\mathbb F}_q(T)$ and $l|q-1$, the analogue is
to consider as the norm group of $K_{H^+}$ the subgroup 
$K^*(\Delta\times \prod_{\pK\nmid\infty} U_{\pK})$ of the
id\`ele group $J_K$ of $K$, where $\Delta=\producto l$
and $U_{\pK}$ denotes the group of units of $K_{\pK}$.
This is the approach taken by Clement.

In \cite{CuMaRz2021}, E. O. Curiel-Anaya, M. R. Maldonado-Ram\'irez
and the first author, use this approach to extend the result to general
cyclic Kummer extensions.

In this note we will see that this approach still works for general Kummer
extensions $K$ of $k$ and obtain a version of extended genus field
for this kind of extensions. In the last section we compare three
definitions of extended genus fields for this kind of extensions
and present some examples.

\section{General Kummer extensions of $k$}\label{S2}

Let $k:={\mathbb F}_q(T)$ and let $K$ be a Kummer extension
of $k$ of exponent $n|q-1$. Then
\[
\Gal(K/k)\cong C_{n_1}\times \cdots \times C_{n_s},
\]
where in general $C_m$ denotes the cyclic group of order $m$
and $n_s|\cdots |n_2|n_1=n$.

We have 
\begin{gather}\label{Ec1}
K=k\big(\sqrt[n_1]{\gamma_1D_1},
\ldots, \sqrt[n_s]{\gamma_sD_s}\big)=K_1 \cdots K_s,
\end{gather}
where $K_{i}=\Ku {n_{i}}{\gamma_{i}
D_{i}}$,  $\gamma_{i}\in \f$,
$D_{i}\in R_T:=\F[T]$ is a monic polynomial and $\Gal(K_{i}/k)
\cong C_{n_i}$, $1\leq i\leq s$.

Let $P_1,\ldots, P_r\in R_T^+:=\{P\in R_T\mid P \text{\ is a
monic irreducible polynomial}\}$ be the primes of $k$
ramified in $K$ with ramification indices $e_1,\ldots, e_r$
respectively. Let $\Delta:=\producto n$, where $\{\pK_i\}_{i=1}^t$
are the infinite primes in $K$, that is, the primes above $\p$
($=\infty$), the infinite prime of $k$ and $k_{\infty}:=k_{\p}$.

Let ${\mc O}_K:=\{\alpha\in K\mid v_{\pK}(\alpha)\geq 0 \text{\ for
all\ } \pK\nmid\infty\}$. Then ${\mc O}_K$ is a Dedekind domain
such that the fractional ideals of ${\mc O}_K$ correspond
canonically to the finite primes of $K$. We denote by $J_K$ the
id\`ele group of $K$.

\begin{definition}\label{D2.1}{\rm{
The {\em extended Hilbert class field} $K_{H^+}$ of $K$
is the finite abelian extension of $K$ such that the finite primes of $K$ fully
decomposed in $K_{H^+}$ are precisely the principal nonzero ideals
generated by an element whose norm from $K$ to $k$ is an $n$-th power
in $k_{\infty}$.

The {\em extended genus field} $\Gamma$ of $K$
with respect to $k$ is the maximal
abelian extension of $k$ contained in $K_{H^+}$.
}}
\end{definition}

\begin{proposition}\label{P2.2}
We have that $[J_K:\*K\big(\A\big)]$ is finite.
\end{proposition}

\begin{proof}
The same proof in \cite{Cle92} works in this case. We present the
proof for the sake of completeness. Let $A:= \A$ and $B:=\B$.
We have
\begin{gather*}
\*K\big(\A\big)=\*K A\subseteq \*KB=\*K\big(\B\big),\\
J_K/\*KB=J_K/\*K\big(\B\big)\xrightarrow[\ \psi\ ]{\ \cong\ } Cl({\mc O}_K)
\end{gather*}
where the isomorphism $\psi$ is the one induced by the natural
map $\theta\colon J_K\longrightarrow D_K$
given by $\theta(\vec \alpha)=\prod_{\pK}\pK^{\alpha_{\pK}}$,
$D_K$ denotes the divisor group of $K$,  and
the ideal class group $Cl({\mc O}_K)$ is isomorphic to
$D_K/D_{\infty}P_K$,
where $D_{\infty}$ is the free group generated by the infinite
primes and $P_K$ is the group of principal divisors. We have
that $Cl({\mc O}_K)$ is finite.

Now $A\subseteq B$. Therefore
\begin{gather*}
\begin{align*}
[\*KB:\*KA]&=\frac{[B:A]}{[\*K\cap B:\*K\cap A]}=\frac{[\B:\A]}{[\*K
\cap B:\*K\cap A]}\\
&=\frac{[\prod_{\pK|\infty}\*{K_{\pK}}:\Delta]}{[U_K:U_K^{(+)}]},
\end{align*}
\intertext{where}
U_K:=\*K\cap B=\{x\in\*K\mid v_{\pK}(x)=0\text{\ for all\ }
\pK\nmid \infty\}=\*{{\mc O}_K}\quad\text{(units of ${\mc O}_K$)}
\intertext{and}
U_K^{(+)}:=\{x\in U_K\mid\N_{K/k}(x)=
\prod_{\pK|\infty}\N_{K_{\pK}/k_{\infty}}(x)\in k_{\infty}^{*n}\}.
\end{gather*}

Now $n|q-1$ so that $\gcd(n,p)=1$ where $p$ is the characteristic
of $k$. Therefore $(U_{\pK}^{(1)})^n=U_{\pK}^{(1)}$. Since $\*{k_{\infty}}=
\langle\pi_{\infty}\rangle\times \f\times U_{\infty}^{(1)}$ where $\pi_{
\infty}=1/T$ is a uniformizing parameter, we have $k_{\infty}^{*n}=
\langle\pi_{\infty}^n\rangle\times {\f}^n\times (U_{\infty}^{(1)})^n=
\langle\pi_{\infty}^n\rangle\times {\f}^n\times U_{\infty}^{(1)}$. Thus
$\big|\frac{\*k_{\infty}}{k_{\infty}^{*n}}\big|=n^2$. Consider
\begin{gather*}
\varphi\colon\prod_{\pK|\infty}\*{K_{\pK}}\lra \frac{\*{k_{\infty}}}{k_{\infty}^{*n}},
\quad \text{given by}\quad
\varphi\big((\alpha_{\pK})_{\pK|\infty}\big)=\Big(\prod_{\pK|\infty}
\N_{K_{\pK}/k_{\infty}}(\alpha_{\pK})\bmod k_{\infty}^{*n}\Big).
\intertext{Then $\ker \varphi=\Delta$ and $\varphi$ induces an injection}
\tilde{\varphi}\colon\frac{\Big(\prod_{\pK|\infty}\*{K_{\pK}}\Big)}{\Delta}
\hookrightarrow \frac{\*{k_{\infty}}}{k_{\infty}^{*n}}.
\end{gather*}
Hence $[B:A]|n^2$ and finally $[J_K:\*K A]<\infty$.
\end{proof}

We have that $\*K\big(\Delta\times \prod_{\pK\nmid \infty}U_{\pK}\big)$
is an open subgroup of finite index of $J_K$. Let $L$ be the class field
corresponding to $\*K\big(\Delta\times \prod_{\pK\nmid \infty}U_{\pK}\big)$.
Then $\N_{L/K}J_L=\*K\big(\Delta\times \prod_{\pK\nmid \infty}U_{\pK}\big)$
and $\Gal(L/K)\cong \frac{J_K}{\*K\big(\Delta\times \prod_{\pK\nmid
\infty}U_{\pK}\big)}$.

\begin{proposition}\label{P2.3}
Let ${\eu q}$ be a finite prime of $K$, that is, ${\eu q}$ is a nonzero prime ideal of
${\mc O}_K$. Then ${\eu q}$ decomposes fully in $L$ if and only if ${\eu q}$ is a
principal ideal with ${\eu q}=\langle\delta\rangle$ and $\N_{K/k}(\delta)\in
k_{\infty}^{*n}$.
\end{proposition}

\begin{proof}
From class field theory, we
have that ${\eu q}$ decomposes fully in $L/K$ if and only if $\theta\lceil
\*{K_{{\eu q}}}\rceil_{{\eu q}} \subseteq \*K\big(\Delta\times \prod_{\pK\nmid \infty}
U_{\pK}\big)$, where $\theta\lceil x \rceil_{{\eu q}}:=(\ldots, 1, x, 1,\ldots)$
for $x\in\*{K_{{\eu q}}}$
(see \cite[Corolario 17.6.196]{RzeVil2017}) if and only if for all $x\in\*{K_{{\eu q}}}$
there exist $\beta_x\in\*K$ and $\vec\alpha_x\in \Delta\times \prod_{\pK\nmid 
\infty}U_{\pK}$ with $\theta\lceil x\rceil_{\eu q}=\beta_x\vec\alpha_x$.
Thus $\vec\alpha_x=(\ldots, \beta_x^{-1},\ldots, \beta_x^{-1},\beta_x^{-1}x,
\beta_x^{-1},\ldots)$. This is equivalent to $v_{\pK}
(\beta_x)=0$ for all $\pK\neq {\eu q}$ and $\pK\nmid\infty$
and $\N_{K/k}(\beta_x^{-1})=\prod_{\pK|\infty}\N_{K_{\pK}/k_{\infty}}(\beta_x^{-1})
\in k_{\infty}^{*n}$.

Let $x\in\*{K_{{\eu q}}}$. The principal ideal in ${\mc O}_K$ generated by
$\beta_x$ satisfies $\langle \beta_x
\rangle={\eu q}^{n_x}=\beta_x{\mc O}_K$, from some $n_x\in{\ma Z}$.
In particular for $x\in\*{K_{{\eu q}}}$ with
$v_{{\eu q}}(\beta_x)=1$ we have $\langle \beta_x\rangle={\eu q}$, $\beta_x\in
\*K$. Hence ${\eu q}$ is a principal ideal, ${\eu q}=\langle \beta_x\rangle$ and
$\N_{K/k}(\beta_x)\in k_{\infty}^{*n}$.
\end{proof}

\begin{corollary}\label{C2.4}
The extended class field $K_{H^+}$ of $K$ is the class field corresponding
to $\*K\big(\Delta\times \prod_{{\eu q}\nmid \infty}U_{{\eu q}}\big)$ and a finite
prime ${\eu q}$ of $K$ decomposes fully in $K_{H^+}$ if and only if ${\eu q}$ is
a principal ideal and the norm from $K$ to $k$ of a generator belongs to
$k_{\infty}^{*n}$. \fin
\end{corollary}

\begin{proposition}\label{P2.5}
The field of constants of $K_{H^+}$ is ${\mathbb F}_{q^n}$.
\end{proposition}

\begin{proof}
We have (see \cite[Teorema 17.6.192]{RzeVil2017}) that the field of constants
of $K_{H^+}$ is ${\mathbb F}_{q^d}$ where $d=\min\{\deg \vec\alpha\mid
\vec\alpha\in \Delta\times\prod_{\pK\nmid\infty}U_{\pK}\text{\ and\ }\deg\vec\alpha
>0\}$. We will see that $d=n$.

On the one hand, since $\prod_{\pK|\infty}
\N_{K_{\pK}/k_{\infty}}\*{K_{\pK}}=k_{\infty}^{*n}$,
there exists $\vec\alpha\in \Delta\times\prod_{\pK\nmid\infty}U_{\pK}$ with
$v_{\infty}\Big(\prod_{\pK|\infty}\N_{K_{\pK}/k_{\infty}}(\alpha_{\pK})\Big)=
v_{\infty}(\pi_{\infty}^n)=n$.
On the other hand, let $\vec\alpha\in
\Delta\times\prod_{\pK\nmid\infty}U_{\pK}$. Then $v_{\pK}(\alpha_{\pK})=0$ for
all $\pK\nmid \infty$ and $\xi=\prod_{\pK|\infty}\N_{K_{\pK}/k_{\infty}}(\alpha_{\pK})
\in k_{\infty}^{*n}$. Therefore (see \cite[Proposici\'on 17.3.38]{RzeVil2017})
\begin{align*}
\deg\vec\alpha&=\sum_{\pK}\deg\pK\cdot v_{\pK}(\alpha_{\pK})=\sum_{\pK|\infty}
\deg\pK \cdot v_{\pK}(\alpha_{\pK})=\sum_{\pK|\infty}f_{\pK}v_{\pK}(\alpha_{\pK})\\
&=\sum_{\pK|\infty}v_{\infty}\big(\N_{K_{\pK}/k_{\infty}}(\alpha_{\pK})\big)=
v_{\infty}\Big(\prod_{\pK|\infty}\N_{K_{\pK}/k_{\infty}}(\alpha_{\pK})\Big)
=v_{\infty}(\xi)=na
\end{align*}
for some $a\in{\mathbb Z}$. Hence $d=n$. 
\end{proof}

\begin{proposition}\label{P2.6}
Every abelian extension $L/k$ such that $K\subseteq L$, $L/K$ is unramified at the
finite primes and such that $\Gal(L/k)$ is of exponent a divisor of $n$, is
contained in $K_{H^+}$.
\end{proposition}

\begin{proof}
Since $K_{H^+}$ is the class field
of $\A$ (Corollary \ref{C2.4}), we have that $K_{H^+}/K$ is unramified at
the finite primes.

Let $L/k$ be an abelian extension with $K\subseteq L$, such that $L/K$ is 
unramified at the finite primes and such that 
$\Gal(L/K)$ is of exponent a divisor of $n$.

Let $\N_{L/K} J_{L}=\*K {\mc H}\subseteq J_K$ for some ${\mc H}$.
If ${\eu q}$ is a finite prime, then ${\eu q}$
is unramified in $L/K$ so that $U_{{\eu q}}\subseteq {\mc H}$
(see \cite[Corolario 17.6.196]
{RzeVil2017}). If $\vec\alpha\in\prod_{\pK\nmid\infty}U_{\pK}$,
and $(\underline{\ },E/F)_{\pK}$ denotes the Artin reciprocity map,
we have (see \cite[Teorema 17.6.149]{RzeVil2017})
\[
(\vec\alpha,L/K)=\prod_{\pK|\infty}(1,L_{\pL}
/K_{\pK})\times \prod_{\pK\nmid\infty}
(\alpha_{\pK},L_{\pL}/K_{\pK})=1,
\]
where $\pL$ is a prime in $L$ above $\pK$. It follows that 
$\vec\alpha\in {\mc H}$ and
$\prod_{\pK\nmid \infty}U_{\pK}\subseteq {\mc H}$. 

If $\vec\alpha\in\Delta$ we have the commutative diagram
\[
\xymatrix{
J_K\ar@{->}[rrr]^{(\underline{\ },L/K)}\ar@{->}[d]_{\N_{K/k}}&&&
\Gal(L/K)\ar@{->}[d]^{\iota}\\
J_k\ar@{->}[rrr]_{(\underline{\ },L/k)}&&&\Gal(L/k)
}
\]
where $\iota$ is the natural embedding. We have
that $\N_{K/k}(\Delta)=\{(\delta,1,\ldots,1,\ldots)\mid \delta\in
k_{\infty}^{*n}\}$. Since $\Gal(L/k)$ is of exponent $n$, we
have that $\N_{K/k}(\Delta)\subseteq \ker ({\underline{\ }},L/k)$
so that $\Delta\subseteq \ker ({\underline{\ }},L/K)=\*K{\mc H}$. Hence
$\Delta\times\prod_{\pK\nmid\infty}U_{\pK}\subseteq \*K {\mc H}$
and $L\subseteq K_{H^+}$.
\end{proof}

\begin{proposition}\label{P2.7}
The extension $K_{H^+}/k$ is a Galois extension.
\end{proposition}

\begin{proof}
It is the same as the one in \cite{Cle92}. It follows from the fact that
$\sigma\big(\Delta\times\prod_{\pK\nmid \infty}U_{\pK}\big)=
\Delta\times\prod_{\pK\nmid \infty}U_{\pK}$ for each $k$--embedding
$\sigma$ of $K_{H^+}$ in an algebraic closure $\bar k$ of $k$.
\end{proof}

The following result is the analogue of Lemma 1.5 and Proposition
1.6 in \cite{Cle92}. The same proof works in our more general setting.

\begin{proposition}\label{P2.8}
Let $I_K$ be the group of fractional ideals of ${\mc O}_K$ and
$P_K$ the subgroup of principal ideals. Let $J_K^{(+)}=
\{\vec\alpha\in J_K\mid (\alpha_{\pK})_{\pK|\infty}\in \Delta\}$, $K^{(+)}=
\*K\cap J_K^{(+)}$ and $P_K^{(+)}=\{\langle\beta\rangle\in P_K\mid
\beta\in K^{(+)}\}$. Then
\begin{gather*}
\frac{J_K}{\*K\big(\A\big)}\cong\frac{J_K^{(+)}}{K^{(+)}
\big(\Delta\times\prod_{\pK
\nmid\infty}U_{\pK}\big)}\cong \frac{I_K}{P_K^{(+)}}=:
Cl({\mc O}_K)^{(+)}. \tag*{$\fin$}
\end{gather*}
\end{proposition}

Next we recall a result from class field theory.

\begin{proposition}\label{P2.9}
Let $E/F$ be a finite extension of global fields. Let $N$
be an open subgroup of finite index of $J_E$ such that $\*E
\subseteq N$. Let $E_N$ be the class field of $N$, that is,
$\N_{E_N/E}(J_{E_N})=N$. Let $F_0$ be the maximal abelian
extension of $F$ contained in $E_N$. Then the norm group
of $F_0$ in $J_F$ is $\*F \N_{E/F} (N)$.
\end{proposition}

\begin{proof}
The norm $\N_{E/F}\colon J_E\lra J_F$ is an open map so
that $\N_{E/F}(N)$ is an open subgroup of $J_F$. Since $N$
is of finite index in $J_E$, we have that $\*F \N_{E/F}(N)$ is
of finite index in $J_F$.
\[
\xymatrix{&&&E_N\ar@{-}[dl]\ar@{<-->}[r]&N\subseteq J_E\\
E\ar@{-}[rr]\ar@{-}[d]&&F_0E\ar@{-}[d]\\
F\ar@{-}[rr]&&F_0\ar@{<-->}[r]&\*F\N_{E/F}(N)
}
\]

To show that $F_0$ is the class field of $\*F\N_{E/F}(N)$, it
suffices to show that the group $\N_{E/F}(N)$ is
contained in any open subgroup
of finite index in $J_F$ such that its class field $F'$ is contained in
$E_N$.

Let $N'$ be an open subgroup of finite index in $J_E$ such that
its class field is $EF'\subseteq E_N$. Since $N\subseteq N'$,
it follows that $\N_{E/F}(N)
\subseteq \N_{E/F}(N')$. We have the commutative diagram
(see \cite[Teorema 17.6.137]{RzeVil2017})
\[
\xymatrix{
J_E\ar@{->}[rrr]^{(\underline{\ },EF'/E)}\ar@{->}[d]_{\N_{E/F}}
&&&\Gal(EF'/E)\ar@{->}[d]^{\rest}\\
J_F\ar@{->}[rrr]_{(\underline{\ },F'/F)}&&&\Gal(F'/F)
}
\]
Therefore $\N_{E/F}(N)$ is contained in the subgroup of $J_F$
corresponding to $F'$.
\end{proof}

As a direct consequence of Proposition \ref{P2.9} we have

\begin{corollary}\label{C2.10}
The extended genus field $\Gamma$ of $K/k$ is the class field
of $\*k\N_{K/k}\big(\Delta\times\prod_{\pK\nmid\infty}U_{\pK}\big)$
and 
\begin{gather*}
[\Gamma:k]=|\Gal(\Gamma/k)|=\big[J_k:\*k\N_{K/k}\big(\Delta\times
\prod_{\pK\nmid \infty}U_{\pK}\big)\big]. \tag*{$\fin$}
\end{gather*}
\end{corollary}

To obtain $\Gamma$, we have to compute $\N_{K/k}\big(\Delta\times
\prod_{\pK\nmid \infty}U_{\pK}\big)$.

First we prove two lemmas.

\begin{lemma}\label{L2.11}
We have $J_k=\*k\big(\*{k_{\infty}}\times \prod_{P\in R_T^+} U_{P}\big)$.
\end{lemma}

\begin{proof}
Let $\vec\beta=\big(\beta_{\infty},\beta_P\big)_{P\in R_T^+}\in J_k$.
Let $Q_1,\ldots,Q_t\in R_T^+$ be the finite primes such that $c_i=v_{Q_i}(
\beta_{Q_i})\neq 0$. We have that $v_P(\beta_P)= 0$ for all $P\notin
\{Q_1,\ldots,Q_t\}$. Let $f\in \*k$ be defined by $f=\prod_{i=1}^t Q_i^{c_i}$
($f=1$ if $t=0$). Then $f^{-1}\vec\beta\in\*{k_{\infty}}\times \prod_{P
\in R_T^+} U_P$. The result is obtained.

Another proof is as follows, the class field $F$ of $\*k\big(\*{k_{\infty}}\times \prod_{P
\in R_T^+} U_P\big)$ is unramified over $k$ and $\p$ decomposes fully. Hence $F$
is an extension of constants and the infinite prime (of degree one) decomposes
fully in $F$. Therefore $F=k$.
\end{proof}

\begin{lemma}\label{L2.12} Let $\pK$ be a finite prime in $K$. Let $P$
be the irreducible polynomial in $R_T$ corresponding to the prime in
$k$ below $\pK$.
\las
\item If $\pK$ is unramified, then $\N_{K_{\pK}/k_P}(U_{\pK})=U_P$.
\item If $\pK$ is ramified with ramification index $e_P$, then
$\N_{K_{\pK}/k_P}(U_{\pK})={\mathbb F}_P^{*e_P}\times U_P^{(1)}$,
where ${\mathbb F}_P$ is the field of constants of $k_P$.
\end{list}
\end{lemma}

\begin{proof}
For any extension $E/F$ of local fields, we have $e_{E/F}=[U_F:\N_{E/F}
(U_E)]$ (see \cite[Corolario 17.3.39]{RzeVil2017}). If $E/F$ is tamely 
ramified, $\N_{E/F}(U_E^{(1)})=U_F^{(1)}$. The result follows.
\end{proof}

\begin{proposition}\label{P2.13}
We have
\[
[\Gamma:k]=\big[J_k:\*k\N_{K/k}\big(\Delta\times\prod_{\pK\nmid\infty}U_{\pK}\big)
\big]=n\prod_{i=1}^r e_i
\]
where $e_i$ is the ramification index of $P_i$ in $K/k$, $1\leq i\leq r$.
\end{proposition}

\begin{proof}
From Lemmas \ref{L2.11} and \ref{L2.12}, we obtain
\begin{align*}
&\big[J_k:\*k\N_{K/k}\big(\A\big)\big]\\ 
&=\Big[\*k\big(\A\big):\*k\Big(k_{\infty}^{*n}
\times\prod_{i=1}^r \big({\mathbb F}_{P_i}^{*e_i}\times U_{P_i}^{(1)}\big)
\times \prod_{\substack{P\in R_T^+\\ P\notin\{P_1,\ldots,P_r\}}}U_P\Big)\Big].
\end{align*}
Set $C:=\A$ and $D:=k_{\infty}^{*n}
\times\prod_{i=1}^r\big({\mathbb F}_{P_i}^{*e_i}\times U_{P_i}^{(1)}\big)
\times \prod_{\substack{P\in R_T^+\\ P\notin\{P_1,\ldots,P_r\}}}U_P$.

Then,
\begin{gather*}
[\*kC:\*kD]=\frac{[C:D]}{[\*k\cap C:\*k \cap D]}
\intertext{and}
\frac CD\cong \frac {\*{k_{\infty}}}{k_{\infty}^{*n}}\times
\prod_{i=1}^r\frac{{\mathbb
F}_{P_i}^{*}\times U_{P_i}^{(1)}}
{{\mathbb F}_{P_i}^{*e_i}\times U_{P_i}^{(1)}}
\times\prod_{\substack{P\in R_T^+\\ P\notin\{P_1,\ldots,P_r\}}}\frac{U_P}{U_P}\cong
C_n^2\times \prod_{i=1}^r C_{e_i}.
\end{gather*}

Consider now $\vec\alpha\in\*k\cap C$. Then $\alpha_P\in U_P$ for all $P\in
R_T^+$. Thus $v_P(\alpha_P)=0$ for all $P\in R_T^+$. It follows that
$v_{\infty}(\alpha_{\infty})=0$ and that $\vec\alpha=x\in \f$.
Thus $\*k\cap C=\f$. Similarly $\*k\cap D=
{\f}^n$. Therefore $[\*k\cap C:\*k\cap D]=n$. It follows that $[\*k C:\*k D]=\frac{
n^2\times\prod_{i=1}^r e_i}{n}=n\cdot\prod_{i=1}^r e_i$.
\end{proof}

We have the main result of this note.

\begin{theorem}\label{T2.14}
Let $K$ be given by {\rm{(\ref{Ec1})}}. Then the extended genus field of $K/k$
is 
\[
\Gamma={\mathbb F}_{q^n}\big(T, \sqrt[e_1]{P_1},\sqrt[e_2]{P_2},
\ldots, \sqrt[e_r]{P_r}\big).
\]
\end{theorem}

\begin{proof}
Let $M={\mathbb F}_{q^n}\big(T, \sqrt[e_1]{P_1},\sqrt[e_2]{P_2},
\ldots, \sqrt[e_r]{P_r}\big)$. Then $[M:k]=n\cdot\prod_{i=1}^r e_i$.
From Proposition \ref{P2.13} we obtain that $[\Gamma:k]=[M:k]$ and
from Proposition \ref{P2.6} we obtain that $M\subseteq K_{H^+}$.
Since $M$ is abelian and $\Gamma$ is the maximal abelian 
extension of $k$ contained in $K_{H^+}$, it follows that
$M\subseteq \Gamma$. Hence $\Gamma=M$.
\end{proof}

\begin{remark}\label{R2.15}{\rm{
We observe that the extended genus field $\Gamma$ of a Kummer
extension $K/k$ is also a Kummer extension and furthermore that the
extended genus field $\Gamma/k$ is $\Gamma$ itself.
}}
\end{remark}

\section{Comparison among various extended genus fields}

Let $K$ be given by (\ref{Ec1}). We have some different definitions of extended
genus fields. We compare three of them.

First we denote by $K_{\mathfrak{gex},\text{clement}}$ the extended genus
field obtained here, that is, $K_{\mathfrak{gex},\text{clement}}=\Gamma=
{\mathbb F}_{q^n}\big(T, \sqrt[e_1]{P_1},\sqrt[e_2]{P_2},
\ldots, \sqrt[e_r]{P_r}\big)$.

In \cite{RaRzVi2019} the extended genus field of $K$ was defined as
$K_{\mathfrak{gex},\text{rarzvi}}=E_{\mathfrak{gex},\text{rarzvi}}K$,
where we have 
$E=k\big(\sqrt[n_1]{(-1)^{\deg D_1}D_1},\sqrt[n_2]{(-1)^{\deg D_2}
D_2},\ldots, \sqrt[n_1]{(-1)^{\deg D_s}D_s}\big)$ and then 
\begin{gather*}
E_{\mathfrak{gex},
\text{rarzvi}}=k\big(\sqrt[e_1]{(-1)^{\deg P_1}P_1},\sqrt[e_2]{(-1)^{\deg P_2}P_2},
\ldots,\sqrt[e_r]{(-1)^{\deg P_r}P_r}\big).
\intertext{Therefore}
K_{\mathfrak{gex},
\text{rarzvi}}=\F\big(T,\sqrt[e_1]{\epsilon_1},\sqrt[e_2]{\epsilon_2},\ldots,
\sqrt[e_r]{\epsilon_r}\big)\big(\sqrt[e_1]{P_1},\sqrt[e_2]{P_2},\ldots,
\sqrt[e_r]{P_r}\big),
\end{gather*}
where $\epsilon_i=(-1)^{\deg P_i}\gamma_i$, $1\leq
i\leq r$. Since $e_i|n$ we have $\F\big(\sqrt[e_1]{\epsilon_1},
\sqrt[e_2]{\epsilon_2},\ldots, \sqrt[e_r]{\epsilon_r}\big)\subseteq
{\mathbb F}_{q^n}$. Therefore $K_{\mathfrak{gex},\text{rarzvi}}\subseteq
K_{\mathfrak{gex},\text{clement}}$.

Next, using the definition of extended genus field given by Angl\`es and Jaulent
in \cite{AnJa2000}, it was obtained in \cite{RzeVil2021} that 
$K_{\mathfrak{gex},\text{angjau}} =FK$ for some subfield of $E_{\mathfrak{gex},
\text{rarzvi}}$. Therefore $K_{\mathfrak{gex},\text{angjau}}\subseteq
K_{\mathfrak{gex},\text{rarzvi}}$.

Finally we have
\[
K_{\mathfrak{gex},\text{angjau}}\subseteq K_{\mathfrak{gex},\text{rarzvi}}
\subseteq K_{\mathfrak{gex},\text{clement}}.
\]

We also remark that while $K_{\mathfrak{gex},\text{clement}}$ is defined just for 
Kummer extensions $K/k$, $K_{\mathfrak{gex},\text{rarzvi}}$ and
$K_{\mathfrak{gex},\text{angjau}}$ are defined for any 
separable finite extension $K/k$.

\begin{examples}\label{Ex5.1}{\rm{
\las
\item Let $K:=k(\sqrt[l]{\gamma P})$ where $l$ is a prime number such that
$l|q-1$, $P\in R_T^+$ with $\deg P<l$ and $\gamma\in \f\setminus (\f)^l$.
Then $K_{\mathfrak{gex},\text{angjau}}=
K_{\mathfrak{gex},\text{rarzvi}}$ (see \cite[Example 5.2]{RzeVil2021}).

\item Let $l$ be a prime number such that $l^2|q-1$ and let $K=k\big(
\sqrt[l^2]{\gamma D}\big)$ where $\gamma \in \f$, $D=P_1^{\alpha_1}
\cdots P_r^{\alpha_r}\in R_T$. Then $K_{\mathfrak{gex},\text{angjau}}=
K_{\mathfrak{gex},\text{rarzvi}}$. (see \cite[Example 5.3]{RzeVil2021}).

\item Now, let $l$ be a prime number such that $l|q-1$ and let 
$K:=k(\sqrt[l]{(-1)^{\deg P} P})$ where $P\in R_T^+$
and $\deg P$ denotes the degree of $P$. Then
$K_{\mathfrak{gex},\text{rarzvi}}=K\subsetneqq K_l=
K_{\mathfrak{gex},\text{clement}}$, where $K_l=K
{\mathbb F}_{q^l}$ denotes
the extension of constants of degree $l$ of $K$.
\end{list}
}}
\end{examples}

\begin{remark}\label{R5.2}{\rm{It is very likely that, when $K/k$
is abelian, we have $K_{\mathfrak{gex},\text{angjau}}=
K_{\mathfrak{gex},\text{rarzvi}}$.
}}
\end{remark}

\end{document}